\documentclass[12pt]{article}

\usepackage{a4wide} 
\usepackage{amssymb,amsmath,array}

\input xy
\xyoption{matrix}\xyoption{arrow}
\def\edge{\ar@{-}}
\def\dedge{\ar@{.}}

\long\def\ignore#1{#1}

\newcommand{\st}{{\rm st}}

\newcommand{\qed}{\hfill $\square$}

\newtheorem{theorem}{Theorem}[section]
\newtheorem{proposition}[theorem]{Proposition}

\newtheorem{definition}[theorem]{Definition}
\newtheorem{lemma}[theorem]{Lemma}

\def\widebar{\overline}

\def\sign{{\rm sign}}

\def\mn{{\mathbb N}}
\def\mz{{\mathbb Z}}

\newcommand{\qs}[1]{\left[#1\right]}
\def\dhom{{\rm Dhom}}

\newcommand{\curlyminor}[1]{\{\!\{{#1} \}\!\}}
\newcommand{\areduced}
{\{\widetilde{a}, \widetilde{a+1},\dots ,\widetilde{a+m-1}\}}
\newcommand{\ireduced}
{\{\widetilde{i}, \widetilde{i+1},\dots ,\widetilde{i+m-1}\}}
\newcommand{\aqminor}
{[\widetilde{a}, \widetilde{a+1},\dots ,\widetilde{a+m-1}]}

\def\qdot{(-q)^{\bullet}} 
\def\oq{\mathcal{O}_{q}}

\def\oqmmn{\oq(M_{m,n})}

\def\oqmn{\oq(M_{n})}

\def\oqgmn{\oq(G(m,n))}

\def\oqgmdashn{\oq(G(m',n))}
\def\oqgtwofour{\oq(G(2,4))}

\def\k{k}

\def\goesto{\longrightarrow}

\def\les{<_s}

\begin{document}

\title{Cyclic orders on the quantum grassmannian}

\author{T H Lenagan and E J Russell\footnote{Some of the results in this paper
will appear in the second author's PhD thesis (Edinburgh). The second author
thanks EPSRC for financial support.}}

\date{}

\maketitle

\abstract{
The quantum grassmannian is known to be a graded quantum algebra
with a straightening law when the poset of generating quantum minors is
endowed with the standard partial ordering. In this paper it is shown that
this result remains true when the ordering is subjected to cyclic shifts. The
method involves proving that noncommutative dehomogenisation is possible at
any consecutive quantum minor. 
} 



\section*{Introduction} 
The quantum grassmannian $\oqgmn$, where $m\leq n$, is the subalgebra of the
quantum matrix algebra $\oqmmn$ generated by the maximal ($m\times m$) quantum
minors (precise definitions are given in Section~\ref{section-basicdefs}).
Two useful tools that have been developed recently to study properties of
these important quantum algebras are the notion of a {\em graded quantum
algebra with a straightening law}, \cite{lr2}, and the notion of {\em
noncommutative dehomogenisation}, \cite{klr}. In \cite{lr2} it was shown that
the quantum grassmannian is a graded quantum algebra with a straightening law
and this fact was then used to study homological properties of the quantum
grassmannian; for example, the quantum grassmannian is AS-Cohen Macaulay.
Noncommutative dehomogenisation is useful for passing properties back and
forth between the quantum grassmannian and quantum matrices; for example,
$\oq(M_{m,n-m})$ is the noncommutative dehomogenisation of $\oqgmn$ at the
right-most maximal quantum minor $[n-m+1,\dots,n]$, see \cite[Corollary
4.1]{klr} and this fact is used to transfer the property of being a graded
quantum algebra with a straightening law from the quantum grassmannian to
quantum matrices and the important quantum determinantal factors in
\cite[Theorem 4.1]{lr2}. 

In this paper we show that the uesfulness of both of these tools can be
extended once one realises that one can study partial orders on the set of
generating quantum minors $\Pi$ of $\oqgmn$ other than the usual one defined
by $[i_1,\dots,i_m]\leq [j_i,\dots,j_m]$ iff $i_l\leq j_l$ for each $1\leq
l\leq n$. 

Indeed, for each $1\leq s\leq n$, one can study the {\em cyclic
order $<_s$} defined on $\{1,\dots,n\}$ by $s<_s s+1<_s\dots<_s n<_s 1
<_s\dots<_s s-1$, and use this order to induce a new partial order $\Pi_s$ on
$\Pi$. We show in Section~\ref{section-cyclicorders} that $\oqgmn$ is a graded
quantum algebra with a straightening law with respect to $\Pi_s$. 

In order to do this we first have to show in Section~\ref{section-dhom} that
noncommutative dehomogenisation is possible at the maximal element of $\Pi_s$,
a so-called {\em consecutive quantum minor}, and that the resulting
noncommutative dehomogenisation is once more $\oq(M_{m,n-m})$. 

Once this result has been established we show that one can pass the
property of being a graded quantum algebra with a straightening law from
quantum matrices to the quantum grassmannian equipped with the partial order 
$\Pi_s$.

\section{Basic definitions}\label{section-basicdefs}

In this section, we will give the basic definitions of the objects that
interest us in this paper and recall several results that we need in our
proofs. 
Throughout, $k$ will denote a base field, $q$ will be a
non-zero element of $k$ and $m$ and $n$ denote positive integers. \\

The quantisation of the coordinate ring of the affine variety $M_{m,n}$ of
$m \times n$ matrices with entries in $k$ is denoted ${\mathcal
O}_q(M_{m,n})$. It is the $k$-algebra generated by $mn$ indeterminates
$x_{ij}$, with $1 \le i \le m$ and $1 \le j \le n$, subject to the relations:
\[
\begin{array}{ll}  
x_{ij}x_{il}=qx_{il}x_{ij},&\mbox{ for }1\le i \le m,\mbox{ and }1\le j<l\le
n\: ;\\ 
x_{ij}x_{kj}=qx_{kj}x_{ij}, & \mbox{ for }1\le i<k \le m, \mbox{ and }
1\le j \le n \: ; \\ 
x_{ij}x_{kl}=x_{kl}x_{ij}, & \mbox{ for } 1\le k<i \le m,
\mbox{ and } 1\le j<l \le n \: ; \\
x_{ij}x_{kl}-x_{kl}x_{ij}=(q-q^{-1})x_{il}x_{kj}, & \mbox{ for } 1\le i<k \le
m, \mbox{ and } 1\le j<l \le n.
\end{array}
\]
To simplify, we write 
$\oqmn$ for ${\mathcal O}_q(M_{n,n})$.
The $m \times n$ matrix
${\bf X}=(x_{ij})$ is called the generic matrix associated with
${\mathcal O}_q(M_{m,n})$.
\\

As is well known, there exists a $\k$-algebra {\em transpose isomorphism}
between ${\mathcal O}_q(M_{m,n})$ and ${\mathcal O}_q(M_{n,m})$, see
\cite[Remark 3.1.3]{lr2}. Hence, from now on, we assume that $m \le n$,
without loss of generality.\\

An index pair is a pair $(I,J)$ such
that $I \subseteq \{1,\dots,m\}$ and $J \subseteq \{1,\dots,n\}$ are subsets
with the same cardinality. Hence, an index pair is given by an integer $t$
such that $1 \le t \le m$ and ordered sets 
$I=\{i_1 < \dots < i_t\} \subseteq \{1,\dots,m\}$
and $J=\{j_1 < \dots < j_t\} \subseteq \{1,\dots,n\}$. To any such index pair
we associate the quantum minor 
\[ 
[I|J] = \sum_{\sigma\in {\mathfrak S}_t}
(-q)^{\ell(\sigma)} x_{i_{\sigma(1)}j_1} \cdots x_{i_{\sigma(t)}j_t} . 
\] 


\begin{definition} -- 
The {\it quantisation of the coordinate ring of the grassmannian of
$m$-dimensional subspaces of $\k^n$}, denoted by $\oqgmn$
and informally referred to as the ($m\times n$) quantum grassmannian is the
subalgebra of $\oqmmn$ generated by the $m \times m$
quantum minors.
\end{definition} 


A maximal (that is, $m\times m$) quantum minor in $\oqmmn$ corresponds to an
index pair $[\{1,\dots,m\}|J]$ with $J=\{j_1,\dots,j_m\}\subseteq
\{1,\dots,n\}$. We call such $J$ {\em index sets} and denote the corresponding
minor by $[J]$ in what follows. Thus, such a $[J]$ is a generator of $\oqgmn$.
The set of all index sets is denoted by $\Pi_{m,n}$, or simply $\Pi$ if no
confusion may arise. As $\Pi_{m,n}$ is in one-to-one correspondence with
the set of all maximal quantum minors of $\oqmmn$, we will often identify these two sets.

When writing down an $m\times m$ quantum minor in $\oqgmn$, we will use the
convention that if a column index $j$ is greater than $n$ then $j$ is to be
read as $j-n$. For example, in $\oqgtwofour$ the minor specified by $[45]$ is
the quantum minor $[14]$. In order to stress
this point, we will use the convention that given any integer $j$ then
$\widetilde{j}$ is the integer in the set $\{1,\dots,n\}$ that is congruent to
$j$ modulo $n$. 

A quantum minor $\ireduced$ is said to be a {\bf consecutive
quantum minor} of $\oqgmn$. Recalling the convention above, we see that there
are four consecutive minors in $\oqgtwofour$: they are $[12], [23], [34]$ and
$[\widetilde{4}\,\widetilde{5}] = [14]$. 
More generally, $\oqgmn$ has $n$ consecutive minors.

Two maximal quantum minors $[I]$ and $[J]$ are said to {\em quasi-commute} if
there is an integer $c$ such that $[I][J] = q^c[J][I]$. Recall that an element
$u$ of a ring $R$ is said to be a {\em normal element} if $uR = Ru$, in which
case $uR$ is a two-sided ideal. The following lemma, first obtained in
\cite[Lemma 3.7]{kroblec}, shows that consecutive quantum minors quasi-commute with all
maximal quantum minors.

\begin{lemma}\label{lemma-normal}

Let $\ireduced$ be a consecutive quantum minor in the quantum grassmannian 
$\oqgmn$.
Then $\ireduced$ quasi-commutes with each of the generating
quantum minors of $\oqgmn$. In particular, each  consecutive quantum minor 
is a normal
element of $\oqgmn$. \qed 
\end{lemma} 

Quantum analogues of the classical Pl\"ucker relations are available and are
stated in the following theorem. 

\begin{theorem}\label{theorem-gqpr}
{\em (Generalised Quantum Pl\"ucker Relations for Quantum
grassmannians)}\\
Let  $J_1,J_2,K \subseteq \{ 1, 2, \ldots
,  n \} $ be such that $|J_1|, |J_2| \leq m$ and 
$|K| = 2m-|J_1| - |J_2| > m$.  Then
\begin{equation*}
\sum_{K' \sqcup K'' = K }^{} (-q)^{ \ell(J_1;K')  +\ell(K';K'') 
+\ell(K'';J_2) } [J_1 \sqcup K' ] [K'' \sqcup J_2 ] = 0,
\end{equation*}
where
$ \ell(I;J) = | \{ (i,j) \in I \times J : i >j \}|$.
\end{theorem}

\begin{proof}\cite[Theorem 2.1]{klr}\qed\\\end{proof} 

Often, when using this result, it is not important to know exactly which power
of $-q$ occurs. In this case, we simply write $\qdot$ to denote the relevant
power of $-q$. 

We will also need a version of the Quantum Muir's Law of Extensible Minors.
This result was first obtained by Krob and Leclerc, \cite[Theorem
3.4]{kroblec}, with a proof involving quasi-determinants. The version below,
which is sufficient for our needs, is taken from \cite[Proposition 1.3]{lr3},
and is adapted for use in the quantum grassmannian.

\begin{proposition} \label{proposition-qmuir} 
Let $I_s, J_s$, for $1\leq s\leq d$, be $m$-element subsets of $\{1,\dots,n\}$
and let $c_s\in\k$ be such that $\sum_{s=1}^d c_s[I_s][J_s]=0$ 
in $\oqgmn$. Suppose that
$P$ is a subset of $\{1,\dots,n\}$ such that $(\cup_{s=1}^d I_s) \cup
(\cup_{s=1}^d J_s)\subseteq P$ and let 
$\widebar{P}$ denote $\{1,\dots,n\}\backslash
P$. Then 
\[
\sum_{s=1}^d c_s[I_s\sqcup \widebar{P}][J_s\sqcup \widebar{P}]=0.
\]
holds in $\oqgmdashn$, where $m' = m+\#\widebar{P}$.\qed 
\end{proposition}

This result is used, for example, 
when it is necessary to write down a commutation relation
between two maximal quantum minors $[I]$ and $[J]$, say. The usefulness of the
result is that one may delete the common members of the index pairs $I$ and
$J$ to establish the commutation relation. It will often be the case that we
then only have to find commutation relations for two minors involving at most
$4$ indices, and here we may use the following 
well-known relations in $\oqgtwofour$  
which can easily be checked from the defining relations of quantum matrices. 
\[[ij][ik] = q[ik][ij],\quad [ik][jk] = q[jk][ik],\quad
\mbox{\rm for $i<j<k$}\]
and
\[
\qs{14}\qs{23}  =  \qs{23}\qs{14}, \quad [12][34] =q^2[34][12], \quad
\qs{13}\qs{24}  =  \qs{24}\qs{13} 
        + \left( q-q^{-1} \right) \qs{14}\qs{23}.\]
        There is also a 
Quantum Pl\"{u}cker relation
$
\qs{12}\qs{34} - q\qs{13}\qs{24} +q^2\qs{14}\qs{23} = 0
$. 
This  Quantum Pl\"{u}cker relation may be rewritten as 
$
\qs{34}\qs{12} - q^{-1}\qs{24}\qs{13} +q^{-2}\qs{23}\qs{14} = 0
$ 
and one can also check that 
$
\qs{13}\qs{24}  = q^2 \qs{24}\qs{13} 
+ \left( q^{-1} -q\right) \qs{12}\qs{34}.
$

\section{Dehomogenisation at a consecutive quantum \newline minor}\label{section-dhom}

Noncommutative dehomogenisation was introduced in \cite{klr} in order to pass
properties back and forth between quantum matrices and the quantum
grassmannian. Here, we recall the basic idea. Let $A$ be an $\mn$-graded
$k$-algebra and let $x$ be a homogeneous normal nonzero divisor sitting in
degree one. Then the Ore localisation at the powers of $x$ exists and is a
$\mz$-graded algebra. The (noncommutative) dehomogenisation of $A$ at $x$ is
defined to be the degree zero part of this localisation, see \cite{klr} for
the details. 

The aim in this section is to show that the dehomogenisation of $\oqgmn$ at
any consecutive quantum minor is isomorphic to $\oq(M_{m,n-m})$. This result
is known for the quantum minor $[n-m+1,\dots,n]$, by \cite[Theorem 4.1]{klr}.
The proof for a general consecutive minor $\aqminor$ follows the same route as
in this theorem, but the technicalities are a little more complicated. First,
we identify a suitable generating set for the dehomogenisation.


\begin{lemma}\label{lemma-generators}

The $k$-algebra 
\[
\dhom(\oqgmn, \aqminor)
\] 
is generated by the elements 
\begin{eqnarray*}
\lefteqn{
\curlyminor{j,a,a+1,\dots,
\widehat{(a+m-i)},\dots, a+m-1}:=
}\\
&&[j,\widetilde{a},\widetilde{a+1},\dots,\widehat{(a+m-i)},\dots,
\widetilde{a+m-1}]
\aqminor^{-1},
\end{eqnarray*}
where $j\in\{1,\dots,n\}\backslash\areduced$ 
and $i\in\{1,\dots,m\}$.
\end{lemma} 


\begin{proof}
Let $A$ be the subalgebra of $\dhom(\oqgmn, \aqminor)$ generated by 
the elements 
$\curlyminor{j,a,a+1,\dots,\widehat{(a+m-i)},\dots,a+m-1}$.

Let $I =\{i_1, i_2,\dots, i_m \}$, with each $i_l\in \{1,\dots,n\}$,  
be an index set such that $I\neq\areduced$.
Suppose that $|I\cap\areduced| = m-t$ for some $1\leq t\leq m$. Certainly,
$\dhom(\oqgmn, \aqminor)$ is generated by 
such $\curlyminor{I}$; so it is enough to
show that each $\curlyminor{I}$ is in $A$. This is done by induction on $t$.

First, consider the case where $t=1$. Then $|I\cap\areduced| = m-1$. Hence,
\[
\curlyminor{I}
= \curlyminor{j,a,a+1,\dots,\widehat{(a+m-i)},\dots,a+m-1}
\]
for some 
$j\in\{1,\dots,n\}\backslash\areduced$
and $1\leq i\leq m$; so that $\curlyminor{I}\in A$, by definition. 

Next, consider $t>1$ and suppose that the result is true for $t-1$. Let $I$ be
an index set with $|I\cap\areduced| = m-t$. 
Choose $c\in I\,\backslash\areduced$. 
We will use the generalised quantum Pl\"ucker relations
of Theorem~\ref{theorem-gqpr} to rewrite the product 
$\aqminor [i_1,i_2,\dots,i_m]$. 

In the notation of Theorem~\ref{theorem-gqpr}, let
$K=\{c\}\sqcup\areduced, J_1=\emptyset$ and 
$J_2=I\backslash\{c\}$. Then, 
\begin{equation*}
\sum_{K' \sqcup K'' = K }\,\qdot [K' ] [K'' \sqcup J_2 ] = 0,
\end{equation*}
where either 
\[
K'= \areduced\quad{\rm and}\quad K''=\{c\}, 
\]
in which case $[K' ] [K'' \sqcup J_2 ]= \aqminor [i_1,i_2,\dots,i_m]$, 
or 
\[
K'=\{c\}\sqcup
\{a,a+1,\dots,\widehat{(a+m-i)},\dots,a+m-1\}\quad{\rm and}
\quad 
K''= \{a+m-i\}
\]
for some $1\leq i\leq m$. Note that in this case, $\widetilde{(a+m-i)}\not\in
I$. Set $S=\{i\mid \widetilde{(a+m-i)}\not\in I\}$. Then, by re-arranging the
above equation, we obtain
\begin{eqnarray*}
\lefteqn{\aqminor [i_1,i_2,\dots,i_m] =}\\ 
&&- \sum_{i\in S}\,\qdot
[c,\,\widetilde{a}, \dots,\widehat{a+m-i},\dots \widetilde{a+m-1}]
[\widetilde{a+m-i},\,i_1,\dots,\widehat{c},\dots,i_m]
\end{eqnarray*}
Multiplying through this equation by $\aqminor^{-2}$ from the right, and using
Lemma~\ref{lemma-normal} gives
\begin{eqnarray*}
\lefteqn{\curlyminor{i_1,i_2,\dots,i_m} = }\\
&&\sum_{i\in S}\,\pm\qdot
\curlyminor{c,\,\widetilde{a}, \dots,\widehat{a+m-i},\dots \widetilde{a+m-1}}
\curlyminor{\widetilde{a+m-i},\,i_1,\dots,\widehat{c},\dots,i_m}
\end{eqnarray*}
Consider the terms on the right hand side of this equation. The first factor
of each term is in $A$ by definition. For the second factor, note that 
\[
|\{\widetilde{a+m-i},\,i_1,\dots,\widehat{c},\dots,i_m\}\cap 
\{\widetilde{a},\widetilde{a+1},\dots \widetilde{a+m-1}\}| 
= m-t+1 = m-(t-1); 
\]
and so 
$\curlyminor{\widetilde{a+m-i},\,i_1,\dots,\widehat{c},\dots,i_m}\in A$, 
by the inductive hypothesis. 
\qed
\end{proof}


\begin{theorem} \label{theorem-dhomiso}
There is an isomorphism 
\[
\rho:\oq(M_{m,n-m}) \goesto \dhom(\oqgmn, \aqminor)
\]
 which is defined on generators by 
 \[
 \rho(x_{ij}) = 
\curlyminor{\widetilde{(j+a+m-1)},\; \widetilde{a}, 
\ldots, \widehat{a+m-i},\ldots, \widetilde{a+m-1}},
\] 
for $1\leq i \leq m$ and
$1 \leq j \leq n-m$. 
\end{theorem} 


\begin{proof}
In order to show that $\rho$ defines a homomorphism, we have to show that the
images of the $x_{ij}$ under $\rho$ obey the relevant commutation relations.
As indicated at the start of Section~\ref{section-basicdefs}, there are four
types of relations to consider. 
Set 
\[
I:=\{ \widetilde{a}, \widetilde{a+1},
\ldots,\ldots, \widetilde{a+m-1}\}
\backslash\{\widetilde{a+m-k},\widetilde{a+m-i}\}.
\] Then
\begin{eqnarray*}
\rho(x_{ij})&=& 
 [\widetilde{(j+a+m-1)},\; \widetilde{a}, 
\ldots, \widehat{a+m-i},\ldots, \widetilde{a+m-1}]\aqminor^{-1}\\
&=& [\widetilde{(j+a+m-1)},\;\widetilde{a+m-k},\,I]
[\widehat{a+m-i}, \widetilde{a+m-k}, I]^{-1}
\end{eqnarray*} 
and 
\begin{eqnarray*}
\rho(x_{kl})&=&
[\widetilde{(l+a+m-1)},\; \widetilde{a}, 
\ldots, \widehat{a+m-k},\ldots, \widetilde{a+m-1}]\aqminor^{-1}\\
&=& [\widetilde{(l+a+m-1)},\; \widetilde{a+m-i},\,I]
[\widehat{a+m-i}, \widetilde{a+m-k}, I]^{-1}
\end{eqnarray*} 
In order to calculate commutation relations  between these two elements, 
we may ignore the occurences of $I$, by using the Quantum Muir's Law, 
Proposition~\ref{proposition-qmuir}. 
This reduces the problem to computations that only involve the (at most) 
four columns
\begin{eqnarray}\label{fourcolumns}
\widetilde{(j+a+m-1)}, \quad\widetilde{(l+a+m-1)},\quad
\widetilde{a+m-i}\quad {\rm and}\quad \widetilde{a+m-k}.
\end{eqnarray}
As only the order of the columns is relevant, 
all the necessary computations can be done by using
the known relations in $\oqgtwofour$. 
For each commutation relation, there are several
subcases involving the relative positions of the column 
indices~(\ref{fourcolumns}).
Here, we present just two calculations, since the computations
are similar in all cases. 
The final type of relation is the most involved, 
so we will concentrate on
that one. So, suppose that $i<k$ and $j<l$. We must show that
\[
\rho(x_{ij})\rho(x_{kl})-\rho(x_{kl})\rho(x_{ij})
=(q-q^{-1})\rho(x_{il})\rho(x_{kj}).
\]
(Note that $\rho(x_{il}) = [\widetilde{(l+a+m-1)},\;\widetilde{a+m-k},\,I]
[\widehat{a+m-i}, \widetilde{a+m-k}, I]^{-1}$ and that 
$\rho(x_{kj})=[\widetilde{(j+a+m-1)},\; \widetilde{a+m-i},\,I]
[\widehat{a+m-i}, \widetilde{a+m-k}, I]^{-1}$.)
The restrictions $1\leq i<k\leq m$ and $1\leq j<l\leq n-m$ ensure that
\[
a+m-k<a+m-i<j+a+m-1<l+a+m-1\]
Thus, one of the following four cases must hold:
\begin{eqnarray}
\widetilde{a+m-k}<\widetilde{a+m-i}< 
\widetilde{(j+a+m-1)}<\widetilde{(l+a+m-1)}\label{first}\\
\widetilde{a+m-i}< 
\widetilde{(j+a+m-1)}<\widetilde{(l+a+m-1)}<\widetilde{a+m-k}\label{second}\\
\widetilde{(j+a+m-1)}<\widetilde{(l+a+m-1)}<\widetilde{a+m-k}
<\widetilde{a+m-i}\label{third}\\
\widetilde{(l+a+m-1)}<\widetilde{a+m-k}
<\widetilde{a+m-i}<\widetilde{(j+a+m-1)}\label{fourth}
\end{eqnarray}

We can check the commutation relations from the commutation relations for
$\oqgtwofour$, since only the ordering of the column indices affects the
relation. Thus, for example, we see that in Case~(\ref{first}) we need to
check that 
$[13][12]^{-1}[24][12]^{-1} 
-[24][12]^{-1}[13][12]^{-1}
= (q-q^{-1})[14][12]^{-1}[23][12]^{-1}$, 
and we now do this:
\begin{eqnarray*}
[13][12]^{-1}[24][12]^{-1} 
-[24][12]^{-1}[13][12]^{-1}
&=& q^{-1}[13][24][12]^{-2} -q^{-1}[24][13][12]^{-2}\\
= q^{-1}([13][24] - [24][13])[12]^{-2}
&=& q^{-1}(q-q^{-1})[14][23][12]^{-2}\\
&=& (q-q^{-1})[14][12]^{-1}[23][12]^{-1},
\end{eqnarray*} 
as required. 

Next, consider Case~(\ref{second}) above. Here we need to check 
the following equality: 
\begin{eqnarray*}
\lefteqn{[24][14]^{-1}[13][14]^{-1} - [13][14]^{-1}[24][14]^{-1} ~~~~~~~~~}\\
&=&
q[24][13][14]^{-2}-q^{-1}[13][24][14]^{-2}\\
&=& (q[24][13] - q^{-1}[13][24])[14]^{-2}\\
&=&
(q[24][13] - q^{-1}[24][13] - q^{-1}(q-q^{-1})[14][23])[14]^{-2}\\
&=&
(q-q^{-1})([24][13]- q^{-1}[14][23])[14]^{-2}\\
&=&
(q-q^{-1})q[34][12][14]^{-2}\\
&=&
(q-q^{-1})[34][14]^{-1}[12][14]^{-1}.
\end{eqnarray*} 
(Note that the fifth equality is obtained by 
using a version of the quantum Pl\"ucker relation.)

The remaining cases to be considered in the verification of the final quantum
matrix relation are similar to, but easier than, the above two cases; 
so we omit the rest of the calculations. 

Thus, $\rho$ extends to a homomorphism. The images of the generators under
$\rho$ generate $\dhom(\oqgmn, \aqminor)$, by Lemma~\ref{lemma-generators}; so
$\rho$ is an epimorphism. 
In order to see that $\rho$ is a monomorphism, we use
Gelfand-Kirillov dimension. The argument is exactly the same 
as that given at the end
of \cite[Theorem 4.1]{klr}.
\qed\end{proof}


\section{The cyclic order $<_s$}
\label{section-cyclicorders}

The set $\Pi = \Pi_{m,n}$ of index sets 
(equivalently, of generating quantum minors of $\oqgmn$) 
carries a natural partial order
defined in the following way. 
Let $I=\{i_1< \dots
<i_m\}$ and $J=\{j_1< \dots <j_m\}$ be  two index sets, then 
\[ 
I\le_\st J
\qquad\Longleftrightarrow\qquad 
i_k \le j_k \quad\mbox{for}\quad 1 \le k \le m. 
\]

In order to study properties of the quantum grassmannian, the notion of a
quantum graded algebra with a straightening law (on a partially ordered set
$\Pi$) was introduced in \cite{lr2}. We now recall the definition of these 
algebras and mention various properties that we will
use later.\\

Let $A$ be an algebra and $\Pi$ a finite subset of 
elements of $A$ with a partial order $<_\st$. A  
{\em standard monomial} on $\Pi$ is an element
of $A$ which is either $1$ or of the form $\alpha_1\cdots\alpha_s$, 
for some $s\geq 1$, where $\alpha_1,\dots,\alpha_s \in \Pi$ and
$\alpha_1\le_\st\dots\le_\st\alpha_s$. 

\begin{definition} \label{recall-q-gr-asl}
Let $A$ be an ${\mathbb N}$-graded $\k$-algebra and $\Pi$ a finite subset 
of $A$ equipped with a partial order $<_\st$. 
We say that $A$ is a {\em quantum graded algebra with a straightening law} 
on the poset $(\Pi,<_\st)$ 
if the following conditions are satisfied.\\
(1) The elements of $\Pi$ are homogeneous with positive degree.\\
(2) The elements of $\Pi$ generate $A$ as a $\k$-algebra.\\
(3) The set of standard monomials on $\Pi$ is a linearly independent set.\\
(4) If $\alpha,\beta\in\Pi$ are not comparable for $<_\st$, 
then $\alpha\beta$ 
is a linear combination of terms $\lambda$ or $\lambda\mu$, where 
$\lambda,\mu\in\Pi$, $\lambda\le_\st\mu$ and $\lambda<_\st\alpha,\beta$.\\
(5) For all $\alpha,\beta\in\Pi$, there exists $c_{\alpha\beta} \in \k^\ast$ 
such that $\alpha\beta-c_{\alpha\beta}\beta\alpha$ is a linear combination of 
terms $\lambda$ or $\lambda\mu$, where $\lambda,\mu\in\Pi$,
$\lambda\le_\st\mu$ and $\lambda<_\st\alpha,\beta$.\\
\end{definition}

By \cite[Proposition 1.1.4]{lr2}, if $A$ is a quantum graded 
algebra with a straightening law on the
partially 
ordered set $(\Pi,<_\st)$, then the set of standard monomials on $\Pi$ forms a
$\k$-basis of $A$. Hence, in the presence of a standard monomial basis, the
structure of a quantum graded algebra with a straightening law 
may be seen as providing more detailed 
information on the way standard monomials multiply and commute.

It is shown, in \cite[Theorem 3.4.4]{lr2}, that $\oqgmn$
is a quantum graded algebra with a straightening law on $(\Pi_{m,n},\le_\st)$.

The aim in this section is to show that there are other partial orderings that
can be put on $\Pi$ in such a way that $\oqgmn$ has the structure of a 
quantum graded algebra with a straightening law. 

Consider the order $\les$ defined by 
$s\les s+1\les\dots\les n\les 1\les\dots\les s-1$.

We use this ordering of the set $\{1,\dots, n\}$ of column indices of $\oqmmn$
to induce a partial ordering $<_s$ on $\Pi =\Pi_{m,n}$:
let $I=\{i_1<_s \dots
<_s i_m\}$ and $J=\{j_1<_s  \dots <_s j_m\}$ be  two index sets, then 
\[ 
I\le_s J \qquad\Longleftrightarrow\qquad  
i_k \le_s j_k \quad\mbox{for~each}\quad
k\in\{1,\dots,m\}. 
\]

When we are considering $\Pi$ with this induced partial ordering, 
we will use the notation 
$\Pi_s$. 

For example, Figure~\ref{figure-2ordering} shows the poset $\Pi_2$ 
in $\oqgtwofour$.

\begin{figure}[ht]
\ignore{
$$\xymatrixrowsep{2.4pc}\xymatrixcolsep{3.2pc}\def\objectstyle{\scriptstyle}
\xymatrix@!0{
 & [14] \edge[d]&\\
 &[13] \edge[dl] \edge[dr]\\
[12]\edge[dr]&&[34]\edge[dl]\\
&[24]\edge[d]&\\
&[23]&
}$$
\caption{The poset $\Pi_2$ on $\oqgtwofour$.}
\label{figure-2ordering}
}
\end{figure}

The aim in this section is to show that $\oqgmn$ is a 
graded quantum algebra
with a straightening law with respect to the poset $\Pi_s$. \\

Set $M=\areduced$ for some $1\leq a\leq n$. 
In the previous section, we have seen that the
dehomogenisation of $\oqgmn$ at $[M]$ is isomorphic to $\oq(M_{m,n-m})$. 
We will show
that the usual standard partial order on the quantum minors of
$\oq(M_{m,n-m})$ is order isomorphic to the partial order  $\Pi_s$ on 
$\oqgmn$ when $a=s-m$. Once this is established,
we use the fact that $\oq(M_{m,n-m})$ is a graded quantum algebra with a
straightening law to obtain the desired result. 

In order to do this, we need to know how the quantum minors of 
$\oq(M_{m,n-m})$ behave under the dehomogenisation isomorphism $\rho$ of 
Theorem~\ref{theorem-dhomiso}.

Note that 
\[
\rho(x_{ij}) = 
\curlyminor{(j+a+m-1),\; a, 
\ldots, \widehat{a+m-i},\ldots, a+m-1}, 
\] 
for $1\leq i \leq m$ and
$1 \leq j \leq n-m$.

Consider the quantum minor $[I|J]$ of $\oq(M_{m,n-m})$. Suppose that $I=\{i_1,
\dots, i_t\}$ and $J=\{j_1, \dots, j_t\}$, for some $1\leq t\leq m$, with
$i_k\in\{1,\dots,m\}$ and $j_k\in\{1,\dots,n-m\}$. Define the maximal quantum
minor $[Q(I,J)]\in\Pi_s$ to be the 
quantum minor with index set $Q(I,J)$ defined by 
\begin{eqnarray*}
\lefteqn{Q(I,J):=
\{\widetilde{(j_1 + a+m-1)}, \widetilde{(j_2 + a+m-1)}
\dots,\widetilde{(j_t + a+m-1)}\}}\\
&&\bigsqcup 
\left(\{\widetilde{a}, \widetilde{a+1}, \dots, \widetilde{a+m-1}\}
\backslash
\{\widetilde{a+m-i_1}, \widetilde{a+m-i_2},\dots,\widetilde{a+m-i_t}\}\right)
\end{eqnarray*}
In the special case where $I=\{i\}$ and $J=\{j\}$, we will write $Q(i,j)$
for $Q(I,J)$. Thus, 
\[
\rho(x_{ij}) = 
\curlyminor{(j+a+m-1),\; a, 
\ldots, \widehat{a+m-i},\ldots, a+m-1} = [Q(i,j)][M]^{-1}. 
\] 
Finally, define 
\[
\curlyminor{Q(I,J)}:= [Q(I,J)][M]^{-1}.
\]
The aim is to show that $\rho([I|J]) = \curlyminor{Q(I,J)}$. 
The main calculation is performed in the following preparatory lemma. \\

Set $\sign(i,j) = \ell(j,i)-\ell(i,j)$; so that 
$\sign(i,j)=\left\{\begin{tabular}{ll} 
$1$,& if $i<j$; \\$0$,& if $i=j$;\\ $-1$,& if $i>j$. 
\end{tabular}\right. $


\begin{lemma}\label{lemma-qijm}
Suppose that $I =\{i_1, i_2,\dots,i_t\}$ and $J = \{j_1, j_2,\dots,j_t\}$ with
$t\leq\min\{m,n-m\}$. Let $M=\areduced$, for some $1\leq a\leq n$. Then 
\begin{eqnarray}\label{ifwecan} 
[Q(I,J)][M] + \sum_{k=1}^t\,
(-q)^{(t-k) -\sign(\widetilde{a+m-i_t},\widetilde{j_k+a+m-1})}
[Q(I\backslash\{i_t\},J\backslash\{j_k\})]
[Q(x_{i_tj_k})]
&=&0~~~~~
\end{eqnarray} 
in $\oq(G(m,n))$. 
\end{lemma}


\begin{proof}
\noindent{\bf Special case:}~~ 
We start by considering the special case where $t=m$ and $n=2m$. 
In this case, $I=J=\{1,\dots, m\}$. 
Thus, $Q(I,J) = \{\widetilde{a+m},\dots,\widetilde{a+2m-1}\}$ and 
$M=\{\widetilde{a},\dots,\widetilde{a+m-1}\}$. \\


\noindent{\bf Special case, subcase 1:}
~~First, consider the case where
$m+1\leq a\leq 2m$, and write $a=m+1+b$, with $0\leq b\leq m-1$. 
Note that 
$\widetilde{k+a+m-1} = b+k$ and 
$\sign(a,\widetilde{k+a+m-1}) = \sign(a,k+b) =-1$,
because $k+b<a$. Also, (\ref{ifwecan}), which is what 
we need to prove, becomes
\begin{eqnarray}\label{special1ifwecan} 
[Q(I,J)][M] +\sum_{k=1}^m\,
(-q)^{m+1-k}
[Q(I\backslash\{m\},J\backslash\{k\})]
\times 
[k+b,\; a+1, \dots, 2m,1,\dots,b] =0 
\end{eqnarray}
The proof uses Theorem~\ref{theorem-gqpr} with $J_1 =\emptyset$.
Thus,  
\begin{equation}\label{j1empty}
\sum_{K' \sqcup K'' = K }^{} (-q)^{\ell(K';K'') 
+\ell(K'';J_2) } [K'] [K'' \sqcup J_2 ] = 0,
\end{equation}
and we set $
K=\{b+1, \dots, b+m\}\sqcup \{a\}$ and 
$J_2=\{1 \dots, b,\,a+1, \dots, 2m\}$. 

There are $m+1$ terms in this sum, corresponding to the choices $K''=\{a\}$
and $K''=\{b+k\}$ for $1\leq k\leq m$. 

When 
$K''=\{a\}$ and  $K' = \{b+1, \dots, b+m\}$ we have 
\begin{eqnarray*}
\ell(K';K'') 
+\ell(K'';J_2) &=&\ell(\{b+1, \dots, b+m\};\{a\}) 
+ \ell(\{a\};\{a+1, \dots, 2m, 1,\, \dots, b\})\\
&=& 0+b=b
\end{eqnarray*}
and so the corresponding term in the sum is 
$(-q)^b[Q(I,J)][M]$.

When $K'' = \{b+k\}$ and $K' = \{b+1, \dots, b+m\}\backslash\{b+k\}\sqcup
\{a\}$ we have 
\begin{eqnarray*}
\ell(K';K'') 
+\ell(K'';J_2) &=&\ell(\{b+1, \dots, b+m\}\backslash\{b+k\}\sqcup
\{a\};\{b+k\})\\&&+ \ell(\{b+k\};\{a+1, \dots, 2m,1, \, \dots, b\})\\
&=& (m+1-k)+b
\end{eqnarray*}
and so the corresponding term in the sum is
$(-q)^{m+1-k+b}Q(I\backslash\{m\},J\backslash\{k\})Q(m,k)$. 

Thus, 
\[ 
(-q)^b[Q(I,J)][M] + \sum_{k=1}^m\,
(-q)^{m+1-k+b}Q(I\backslash\{m\},J\backslash\{k\})Q(m,k)=0. 
\] 
Cancelling $(-q)^b$ gives
(\ref{special1ifwecan}), the equality we need to finish this case. 
 \\
     


\noindent{\bf Special case, subcase 2:}
~~Now, consider the case where $1\leq
a\leq m$. Note that $k+a+m-1\leq 2m$ when $k\leq m-a+1$ while $k+a+m-1>2m$
when $k>a+m-1$. Thus, 
$\widetilde{k+a+m-1} = k+a+m-1$ for $k\leq m-a+1$ and 
$\widetilde{k+a+m-1} = k+a-m-1$ for $k>a+m-1$. 
Set $\widebar{k} = \widetilde{k+a+m-1}$ in each of these cases.

Now, 
$\sign(a,\widetilde{k+a+m-1}) = \sign(a,\widebar{k}) =1$
when $k\leq m-a+1$ and, similarly, 
$\sign(a,\widetilde{k+a+m-1}) = -1$ when $k>m-a+1$

Thus, in this case, 
(\ref{ifwecan}), which is what 
we need to prove, becomes
\begin{eqnarray}\label{special2ifwecan} 
\lefteqn{[Q(I,J)][M] \quad +}\notag\\
&&\sum_{k=1}^{m-a+1}\,
(-q)^{m-1-k}
[Q(I\backslash\{m\},J\backslash\{k\})]
[k+a+m-1,\; a+1, \dots, a+m-1]\notag\\
&&-\;
(\sum_{k>m-a+1}^{m}\,
(-q)^{m+1-k}
[Q(I\backslash\{m\},J\backslash\{k\})]
[k+a-m-1,\; a+1, \dots, a+m-1])\notag\\ &&=0 
\end{eqnarray}

The proof again uses Theorem~\ref{theorem-gqpr} with $J_1 =\emptyset$, 
but  $K=\{1,\dots,a-1,a+m,\dots,2m\}\sqcup\{a\}$ and $J_2
=\{a+1,\dots, a+m-1\}$. When $K''=\{a\}$ and $K'=
\{1,\dots,a-1,a+m,\dots,2m\}$ we have
\begin{eqnarray*}
\ell(K';K'') 
+\ell(K'';J_2) &=&\ell(\{1,\dots,a-1,a+m,\dots,2m\};\{a\}) \\
&&+\; \ell(\{a\};\{a+1, \dots, a+m-1\})\\
&=& (m+1-a) + 0 = m+1-a
\end{eqnarray*}
and so the corresponding term in the sum is 
$(-q)^{m+1-a}[Q(I,J)][M]$.

Consider the case that $1\leq k\leq m-a+1$. In this case, $\widebar{k}
=k+a+m-1$ and $a+m\leq \widebar{k}\leq 2m$. When $K'' = \widebar{k}$ and
$K'=\{1,\dots,a-1,a+m,\dots,2m\}\backslash\{\widebar{k}\}\sqcup\{a\}$ we have
\begin{eqnarray*}
\ell(K';K'') 
+\ell(K'';J_2) 
        &=&
\ell(\{1,\dots,a-1,a+m,\dots,2m\}
\backslash\{\widebar{k}\}\sqcup\{a\};\{k+a+m-1\})\\
       && +\;
  \ell(\{k+a+m-1\}; \{a+1,\dots, a+m-1\}) \\
        &=&
2m-(k+m+a-1) + m-1 = 2m-a-k
\end{eqnarray*}
and so the corresponding term in the sum is
\[
(-q)^{2m-a-k}[Q(I\backslash\{m\},J\backslash\{k\})]
[k+a+m-1,\; a+1, \dots, a+m-1]).
\]

Next, consider the case where $m-a+1<k\leq m$. In this case, 
$\widebar{k} = k+a-m-1$ and $1\leq \widebar{k}\leq a-1$. When 
$K'' = \widebar{k}$ and
$K'=\{1,\dots,a-1,a+m,\dots,2m\}\backslash\{\widebar{k}\}\sqcup\{a\}$ 
we have
\begin{eqnarray*}
\ell(K';K'') 
+\ell(K'';J_2) 
        &=&
\ell(\{1,\dots,a-1,a+m,\dots,2m\}
\backslash\{\widebar{k}\}\sqcup\{a\};\{k+a-m-1\}) \\ 
       && +\;
\ell(\{k+a-m-1\}; \{a+1,\dots, a+m-1\})\\
&=&             
m+1-\widebar{k} +   0=     m+1-(k-a-m-1) = 2(m+1) -k -a; 
\end{eqnarray*}
and so the corresponding term in the sum is
\[
(-q)^{2(m+1)-a-k}[Q(I\backslash\{m\},J\backslash\{k\})]
[k+a+m-1,\; a+1, \dots, a+m-1]).
\] 
Thus, 
\begin{eqnarray*}
(-q)^{m+1-a}[Q(I,J)][M] &+& \sum_{k=1}^{m-a+1}\,
(-q)^{2m-a-k}[Q(I\backslash\{m\},J\backslash\{k\})]
[Q(x_{mk})]\\ &+& 
\sum_{k>m-a+1}^{m}\,(-q)^{2(m+1)-a-k}[Q(I\backslash\{m\},J\backslash\{k\})]
[Q(x_{mk})] =0
\end{eqnarray*}

Cancelling $(-q)^{m+1-a}$ gives
(\ref{special2ifwecan}), the equality we need to prove to finish 
this case. \\


This establishes the special case.\\


\noindent{\bf General case:}~~Now, consider the general case. 
Here, the proof is by induction. The base case
of $\oq(G(1,2))$ is trivial to check. First, suppose that the result holds in
$\oq(G(m',n'))$ for all $m'\leq n'<n$. Next, suppose that the result holds in
all $\oq(G(m',n))$ for all $m'<m$. Finally, suppose that the result holds in
$\oqgmn$ for all values of $t'<t$. \\

Suppose that $t<n-m$. Then $t+m<n$; and so there is an index $c$, say, with
$c\not\in M\sqcup\{j_1,\dots,j_t\}$. Note that the index $c$ does not occur in
any of the terms in (\ref{ifwecan}). Thus, we may ignore the column $c$ and
work in $\oq(G(m,n-1))$ where the result holds by the inductive hypothesis. 
\\

Next, suppose that $t=n-m <m$. Choose  an index
$r\in\{1,\dots,m\}\backslash\{i_1,\dots,i_t\}$. The index $\widetilde{a+m-r}$
occurs in each of the quantum minors in (\ref{ifwecan}). By the inductive
hypothesis, the result (\ref{ifwecan}) holds for the triple
$I':=I\backslash\{\widetilde{a+m-r}\},J':=J\backslash\{\widetilde{a+m-r}\},
M':=M\backslash\{\widetilde{a+m-r}\}$ in the copy of $\oq(G(m-1,n-1))$
that sits inside the copy of $\oq(M_{m-1,n-1})$ obtained by removing the row $r$
and the column $\widetilde{a+m-r}$: call the resulting equation ($1'$). 
We obtain the desired result by invoking the Quantum Muir Law, 
Proposition~\ref{proposition-qmuir}, 
to insert the index
$\widetilde{a+m-r}$ in each quantum minor occuring in ($1'$). 
\\


It only remains to consider the case where $t=n-m=m$. However, this is the
special case that was established in the first part of the proof. 
\qed
\end{proof}


\begin{proposition}
$\rho([I|J]) = \curlyminor{Q(I,J)}$
\end{proposition}

\begin{proof}
The proof is by induction on $t$. The case $t=1$ is given in 
Theorem~\ref{theorem-dhomiso}. 

Suppose 
that $I =\{i_1, i_2,\dots,i_t\}$ and $J = \{j_1, j_2,\dots,j_t\}$, 
with $t\geq 2$.
Expand $[I|J]$ along its final row, by using \cite[Corollary 4.4.4]{pw},
to obtain 
\[
[I|J] = \sum_{k=1}^t\,(-q)^{t-k}[I\backslash\{i_t\}|J\backslash\{j_k\}]
x_{i_tj_k}.
\]

Now apply $\rho$ to this expression, using the inductive hypothesis on the
quantum minors $[I\backslash\{i_t\}|J\backslash\{j_k\}]$ to obtain
\begin{eqnarray*}
\rho([I|J]) &=& \sum_{k=1}^t\,(-q)^{t-k}
[Q(I\backslash\{i_t\},J\backslash\{j_k\})][M]^{-1}
[Q(i_t,j_k)]
[M]^{-1}
\end{eqnarray*}

Note that the index sets 
$Q(i_t,j_k)= \{\widetilde{(j_k+a+m-1)},\; \widetilde{a}, \ldots,
\widehat{a+m-i_t},\ldots, \widetilde{a+m-1}\}$ 
and 
$M = \{\widetilde{a},
\widetilde{a+1}, \dots, \widetilde{a+m-1}\}$ differ only in the indices 
$\widetilde{(j_k+a+m-1)}$ and $\widetilde{a+m-i_t}$; so that 
\[
[M]^{-1}
[Q(i_t,j_k)]
=
q^{
-\sign(\widetilde{a+m-i_t},\widetilde{j_k+a+m-1})
}
[Q(i_t,j_k)][M]^{-1}. 
\]
Thus, 
\begin{eqnarray}\label{rhoij}
\rho([I|J]) &=&-(\sum_{k=1}^t\,
(-q)^{(t-k) -\sign(\widetilde{a+m-i_t},\widetilde{j_k+a+m-1})}
[Q(I\backslash\{i_t\},J\backslash\{j_k\})]
[Q(x_{ij})])[M]^{-2}\notag 
\end{eqnarray}
However, 
\[
-(\sum_{k=1}^t\,
(-q)^{(t-k) -\sign(\widetilde{a+m-i_t},\widetilde{j_k+a+m-1})}
[Q(I\backslash\{i_t\},J\backslash\{j_k\})]
[Q(x_{ij})])
=
[Q(I,J)][M]
\]
by Lemma~\ref{lemma-qijm}; so
\[
\rho([I|J]) = [Q(I,J)][M][M]^{-2} = [Q(I,J)][M]^{-1} = \curlyminor{Q(I,J)}
\]
as required.
\qed\\\end{proof}

Recall from Section~\ref{section-basicdefs} the definition of an index pair
$(I,J)$ and the corresponding quantum minor $[I\mid J]$ in a fixed quantum
matrix algebra, say $\oq(M_{m,n-m})$. Let $\Delta$ denote the set of index
pairs (or quantum minors).

We put a partial order on $\Delta$ that is denoted by 
$\leq_\st$. Let $u,v$ be integers such that $1 \leq u \leq m$ and 
$1 \leq v \leq n-m$, and  
let 
$(I,J)$ and
$(K,L)$ be index pairs with
$I=\{i_1< \dots <i_u\}, K=\{k_1< \dots <k_v\} \subseteq \{1,\dots,m\}$,
and 
$J=\{j_1< \dots <j_u\},  L=\{l_1< \dots <l_v\} \subseteq \{1,\dots,n-m\}$.  
We define $\leq_\st$ as follows:
\[
(I,J) \leq_\st (K,L) 
\qquad\Longleftrightarrow\qquad 
\left\{
\begin{array}{l}
u \ge v, \cr 
i_s \leq k_s \quad\mbox{for}\quad 1 \leq s \leq v , \cr
j_s \leq l_s \quad\mbox{for}\quad 1 \leq s \leq v .
\end{array}
\right.
\]
In \cite[Theorem 3.5.3]{lr2} it is shown that quantum matrices form a graded
algebra with a straightening law with respect to this order.

Let $[M]= \aqminor$. The previous proposition shows that for each quantum minor
$[I\mid J]$ of $\oq(M_{m,n-m})$ produces, in a natural way, a generating minor
$[Q(I,J)] = \rho([I\mid J])[M]$ of $\oqgmn$. It is easy to check that every
generating minor of $\oqgmn$, apart from $[M]$ 
itself, arises in this way. Thus,
we can use the previous proposition to induce a partial order on $\Pi$, the
set of generating minors of $\oqgmn$. The following combinatorial lemma
identifies this partial order.


\begin{proposition}\label{proposition-st-versus-s}
~~Let $1\leq s\leq n$ and set $a = \widetilde{s}$. Then 
$[I|J]\leq_\st[K|L]$ if and only if $Q(I,J)\les Q(K,L)$. 
\end{proposition}


\begin{proof}
This is similar to the proof of \cite[Lemma 4.9]{bv}
\qed\\
\end{proof}


Note that $[M] =\aqminor$ is the maximal element in the partially ordered set
$\Pi_s$. Figure~\ref{figure-inducedpo} illustrates the previous result 
in $\oqgtwofour$ with
$s=2$.

\begin{figure}[ht]
\ignore{
{\hskip 10ex}
 \parbox{25ex}{
$$\xymatrixrowsep{2.4pc}\xymatrixcolsep{3.2pc}\def\objectstyle{\scriptstyle}
\xymatrix@!0{
 & 1 \edge[d]&\\
 &[2\mid 2] \edge[dl] \edge[dr]\\
[2\mid 1]\edge[dr]&&[1\mid 2]\edge[dl]\\
&[1\mid 1]\edge[d]&\\
&[12\mid 12]&
}
$$
\begin{center}{$\leq_\st$ on  $\oq(M_2)$}\end{center}
}{\hskip 5ex}$\leadsto\leadsto${\hskip 5ex}
 \parbox{25ex}{
$$\xymatrixrowsep{2.4pc}\xymatrixcolsep{3.2pc}\def\objectstyle{\scriptstyle}
\xymatrix@!0{
 & [14] \edge[d]&\\
 &[13] \edge[dl] \edge[dr]\\
[12]\edge[dr]&&[34]\edge[dl]\\
&[24]\edge[d]&\\
&[23]&
}$$
\begin{center}{$\Pi_2$ on $\oqgtwofour$}\end{center}
}}\caption{}
\label{figure-inducedpo}
\end{figure}

We use the previous results to transfer the graded algebra with a
straightening law property from $\oq(M_{m,n-m})$ to $\oqgmn$. The proof is
essentially obtained by reversing the direction of the proof of \cite[Theorem
3.5.3]{lr2}, and, for this reason, we merely sketch the proof.


\begin{theorem}
The quantum grassmannian $\oqgmn$ is a graded quantum algebra with a
straightening law on the poset $\Pi_s$ for each $1\leq s\leq n$. 
\end{theorem}


\begin{proof} There are five conditions in the definition of a 
graded quantum algebra with a straightening law, see
Definition~\ref{recall-q-gr-asl}. Conditions (1) and (2) are immediate; so we
need to check (3), (4) and (5). We use Theorem~\ref{theorem-dhomiso} with 
$a= \widetilde{s-m}$. 

The map $\rho$ of Theorem~\ref{theorem-dhomiso} extends to an isomorphism \[
\rho:\oq(M_{m,n-m})[y,y^{-1};\sigma] \goesto \dhom(\oqgmn, \aqminor) \] with
$\rho(y)=[M]$, cf. \cite[Corollary 4.1]{klr}. Let $\theta$ denote the inverse of
this isomorphism. Note that $y$ quasi-commutes with each of the quantum minors
in $\oq(M_{m,n-m})$. 

Suppose that $[I_1]^{a_1}[I_2]^{a_2}\dots[I_t]^{a_t}[M]^a$ is a standard
monomial with respect to the ordering $<_s$, and suppose that $I_t\neq M$. Let
$\rho([K_i\mid L_i]) = [I_i][M]^{-1}$ for each $i=1,\dots,t$. Then \[
\theta([I_1]^{a_1}[I_2]^{a_2}\dots[I_t]^{a_t}[M]^a) = \qdot[K_1\mid
L_1]^{a_1}[K_2\mid L_2]^{a_2} \dots[K_t\mid L_t]^{a_t}y^{a+\sum a_i}. \] Note
that this image is a non-zero scalar multiple of a term in the standard basis
of $\oq(M_{m,n-m})$ multiplied by a power of $y$. Note also that distinct
$[I_1]^{a_1}[I_2]^{a_2}\dots[I_t]^{a_t}[M]^a$ produce distinct images. Thus, a
linear combination of such terms is mapped to a linear combination of terms
which are linearly independent, and so the standard monomials with respect to
the ordering $<_s$ are linearly independent. This establishes (3). 

Next, suppose that $[I], [J]$ are incomparable with respect to $<_s$. Note
that neither $[I]$ nor $[J]$ is equal to $[M]$, since $[M]$ is the maximal
element of the poset $\Pi_s$. Thus, there are quantum minors $[K\mid L],
[U\mid V]$ with $\rho([K\mid L]) = [I][M]^{-1}$ and $\rho([U\mid V]) =
[I][M]^{-1}$. Note that $[K\mid L]$ and $[U\mid V]$ are incomparable, by
Proposition~\ref{proposition-st-versus-s}. As $\oqmmn$ is a graded quantum
algebra with a straightening law, there is an equation
\[
[K\mid L][U\mid V] =\sum\,\alpha_i[K_i\mid L_i][U_i\mid V_i]
\]
with $\alpha_i\in k$ and $[K_i\mid L_i]<_{\st}[U_i\mid V_i]$ while
$[K_i\mid L_i]<_{\st}[K\mid L], [U\mid V]$. 

Apply $\rho$ to this equation, and cancel $[M]^{-2}$ 
in the resulting equation
to obtain an equation
\[
[I][J]=\sum\,\alpha_i\qdot[I_i][J_i]
\]
and note that $[I_i]<_s[J_i]$ and $[I_i]<_s [I],[J]$ for each $i$, by using 
Proposition~\ref{proposition-st-versus-s}. This establishes (4). 

Finally, suppose that $[I], [J] \in\Pi_s$. If $[I]= [M]$ or $[J]=[M]$ then
these quantum minors quasi-commute; and so (5) is established for this pair.
Otherwise, we argue in a similar manner to the previous paragraph, but this
time using the fact that (5) holds in quantum matrices, to establish (5) for
$\oqgmn$. 

Thus, $\oqgmn$ is a graded quantum algebra with a straightening law with
respect to the poset $\Pi_s$. 
\qed \end{proof} 



\vskip 1cm
\noindent 
T H Lenagan: 
\\
Maxwell Institute for Mathematical Sciences,\\
School of Mathematics, University of Edinburgh,\\
James Clerk Maxwell Building, 
King's Buildings, 
Mayfield Road,\\
Edinburgh EH9 3JZ, Scotland, UK\\[2ex]
tom@maths.ed.ac.uk\\~\\
E J Russell:\\
Faculty of Engineering and Computing\\
Coventry University\\ Coventry CV1 5FB
\\[2ex]
ewan.russell@coventry.ac.uk 
\end{document}